\documentclass[12pt]{amsart}
\usepackage{graphicx}
\usepackage{a4wide}
\usepackage{amssymb}
\usepackage{amsthm}
\usepackage{amsmath}
\usepackage{amscd}
\usepackage{verbatim}
\usepackage[all]{xy}
\usepackage{url}
\theoremstyle{plain}
\newtheorem{theorem}{Theorem}[section]
\newtheorem{corollary}[theorem]{Corollary}

\newtheorem{proposition}[theorem]{Proposition}

\theoremstyle{definition}

\theoremstyle{remark}

\newcommand{\calF}{\mathcal{F}}

\newcommand{\fraka}{\mathfrak a}
\newcommand{\frakb}{\mathfrak b}

\newcommand{\frakf}{\mathfrak f}

\newcommand{\frakp}{\mathfrak p}

\newcommand{\la}{\lambda}

\newcommand{\La}{\Lambda}
\newcommand{\Z}{\mathbb{Z}}
\newcommand{\A}{\mathbb{A}}
\newcommand{\N}{\mathbb{N}}
\newcommand{\C}{\mathbb{C}}
\newcommand{\K}{\mathbb{K}}

\newcommand{\F}{\mathbb{F}}

\begin{document}

\title[Legendre Drinfeld modules and universal supersingular polynomials]
{Legendre Drinfeld modules and universal supersingular polynomials}

\author{Ahmad El-Guindy}
\address{Current address: Science Program, Texas A\&M University in Qatar, Doha, Qatar}
\address{Permanent address: Department of Mathematics, Faculty of Science, Cairo University, Giza, Egypt 12613}
\email{a.elguindy@gmail.com}

\thanks{}
\date{}

\begin{abstract}
We introduce a certain family of Drinfeld modules that we propose as analogues of the Legendre normal form elliptic curves. We exhibit explicit formulas for a certain period of such Drinfeld modules as well as formulas for the supersingular locus in that family,  establishing a connection between these two kinds of formulas. Lastly, we also provide a closed formula for the supersingular polynomial in the $\jmath$-invariant for generic Drinfeld modules. 
\end{abstract}

\maketitle

\section{Introduction}

Drinfeld modules are function field analogues of elliptic curves that were introduced and studied by Drinfeld in the 70's. They have a rich arithmetic theory, and indeed a number of special cases were already studied by Carlitz and Wade (see \cite{Carlitz35, Wade46} for instance) prior to Drinfeld's general definition. Work of many authors (such as Drinfeld, Hayes, Goss, Gekeler, Anderson, Brownawell,   Thakur, Papanikolas and Cornelissen, to name a few) established many striking similarities, and yet a few astounding differences between elliptic curves and Drinfeld modules (particularly in rank 2); the reader may consult \cite{GossBook} and \cite{ThakurBook} for a more elaborate discussion. The comparison of similarities and differences in the statements (and proofs) between corresponding results in the classical and Drinfeld theories provides valuable insights on some common themes in number theory. One motivation of the present paper is to explore the parallels between the two theories in studying the relation between the \emph{periods} and \emph{supersingular polynomial}.


We start by reviewing some classical facts from the theory of elliptic curves. Recall that for $\la\in \C\setminus\{0,1\}$, the Legendre normal form elliptic curve $E(\la)$ is given by
\begin{equation}\label{Leg}
E(\la): y^2=x(x-1)(x-\la).
\end{equation}

It is well known (for example, see \cite{Hus}) that $E(\la)$ is isomorphic, as a Riemann surface,  to the complex torus $\C/L_\la$, where $L_\la$ is the lattice $\Z\omega_1(\la)+\Z\omega_2(\la)$, and the \emph{periods} $\omega_1(\la)$ and $\omega_2(\la)$ are given by the integrals

\[
\omega_1(\la)=\int_{-\infty}^0 \frac{dx}{\sqrt{x(x-1)(x-\la)}} \ \ \ \textrm{and} \ \
\omega_2(\la)=\int_1^\infty \frac{dx}{\sqrt{x(x-1)(x-\la)}}.  
\]
These integrals can be expressed in terms of Gauss's hypergeometric function

\begin{equation}
\, _2F_1(x):=\, _2F_1\left(\begin{matrix} \frac{1}{2},&
\frac{1}{2}\\ \ \ & 1\end{matrix}\ ; \ x\right)=\sum_{n=0}^\infty \left(\frac {(\frac 12)_n}{(n!)}\right)^2 x^n,
\end{equation}

where $(a)_n=a\cdot(a+1)\cdots(a+n-1)$.  

More precisely, for $\la \in \C\setminus\{0,1\}$ with $|\la|, |\la-1|<1$, we have

\begin{equation}\label{hyperg}
\omega_1(\la)=i \pi \, _2F_1(1-\la) \ \ \ \textrm{and}\ \ \omega_2(\la)=\pi \, _2F_1(\la).
\end{equation}

Next, let $p\geq 5$ be a prime. Recall that an elliptic curve in characteristic $p$ is said to be \emph{supersingular} if it has no $p$-torsion over $\overline{\F}_p$. There are only finitely many (up to isomorphism) supersingular elliptic curves over $\overline{\F}_p$, and the characteristic $p$ \emph{Hasse invariant} for the Legendre normal form elliptic curves is defined by 
\begin{equation}\label{Hp}
H_p(x):= \prod_{\substack{\la\in \overline{\F}_p \\
E(\la)\ {\text {\rm supersingular}}}}
(x-\la).
\end{equation}
It turns out that $H_p(x)$ is in $\F_p[x]$, and it satisfies (for example, see page 261 of \cite{Hus})
\begin{equation}\label{LegendreSS}
H_p(x) \equiv \,  \sum_{n=0}^{\frac{p-1}{2}}
\left(\frac{(\frac{1}{2})_n
}{n!}\right)^2x^n \pmod p.
\end{equation}
Notice that the Hasse invariant is given as a ``truncation" of the expression for the real period. Similar results were obtained for other families of elliptic curves where the supersingular locus corresponds to zeros of a truncation of a period (see \cite{ElgOno} for instance). 

It is natural to investigate whether a similar phenomenon exists in the theory of Drinfeld modules. We start by recalling the basic definition and fixing some notation (for more details the reader can consult \cite{GossBook} for instance). Let $q$ be a power of a prime. Consider the polynomial ring $\A:=\F_q[T]$, and let $\K$ denote the fraction field of $\A$. Consider the unique valuation on $\K$ defined by
\[
v(T)=-1,
\]
which is the valuation at the ``infinite prime'' of the ring $\A$. Let $\K_\infty$ denote the completion of $\K$ with respect to $v$, and let $\C_\infty$ denote the completion of an algebraic closure of  $\K_\infty$. It is well known that $v$ has a unique extension to $\C_\infty$ that we still denote by $v$, and that $\C_\infty$ is a complete algebraically closed field. 
A field $L$ is called an \emph{$\A$-field} if there is a nonzero homomorphism $\imath:{\A}\to L$. Examples of such fields are  extensions of either $\K$ or $\A/\mathfrak{p}$, where $\mathfrak{p}$ is a nonzero prime of $\A$. For simplicity we will write $a$ in place of $\imath(a)$ when the context is clear.  Such a field has a \emph{Frobenius homomorphism}
\[
\begin{split}
\tau:L&\to L\\
z&\mapsto z^q,
\end{split}
\]
and we can consider the ring $L\{\tau\}$ of polynomials in $\tau$ under addition and composition; thus $\tau \ell=\ell^q\tau$ for any $\ell \in L$. A \emph{Drinfeld module of rank $2$ over $L$} is an $\F_q$-linear ring homomorphism $\phi:\A\to L\{\tau\}$ given by
\begin{equation}\label{phi}
\phi_T=T+g\tau+\Delta \tau^2, \, \, \, \Delta\neq 0.
\end{equation}
It then follows that the constant term of $\phi_a$ is $a$ for all $a \in \A$ and that the degree of $\phi_a$ in $\tau$ is $2\deg_T(a)$. Also, the $\jmath$-invariant of $\phi$ is defined by $\jmath(\phi)=\frac{g^{q+1}}{\Delta}$. Two Drinfeld modules are isomorphic over $\overline{L}$ (a fixed algebraic closure of $L$) if and only if they have the same $\jmath$-invariant.

The set $\ker(\phi_a):=\{x\in \overline{L}:\phi_a(x)=0\}$ constitutes the \emph{$a$-torsion points} of $\phi$; it is in fact an $\F_q$-vector subspace of $\overline{L}$. If $L$ is an extension of $\A$ then $\dim_{\F_q}(\ker(\phi_a))=2\deg(a)$, whereas if $\frakp$ is a nonzero prime of $\A$ and $L$ is an extension of $\F_\frakp:=\A/\frakp$, then either $\dim_{\F_q}(\ker(\phi_\frakp))=\deg \frakp$, in which case $\phi$ is called \emph{ordinary at $\frakp$}, or else  $\ker(\phi_\frakp)=0$, in which case $\phi$ is called \emph{supersingular at $\frakp$}. Supersingularity of $\phi$ depends only on its $\overline{\F}_\frakp$-ismorphism class, and hence only on $\jmath(\phi)\in \overline{\F}_\frakp$. 

We now give a brief outline of our results. We shall consider the special family of Drinfeld modules given by
\begin{equation}\label{LegDrin}
\phi_T=T-(T+\Delta)\tau+\Delta\tau^2,
\end{equation}
which we propose as an analogue of the Legendre form elliptic curves. In section \ref{period} we recall results from \cite{ElgPap} that provide a general analytic expression of a period of a Drinfeld module as a certain infinite series. We then specialize the computation to the Legendre Drinfeld modules to obtain a closed form combinatorial expression for the coefficients of said series. In section \ref{tower} we relate these coefficients to certain polynomials $p_n(x)$ that were recently studied by Bassa and Beelen in connection with the Drinfeld modular tower $X_0(T^n)$ over $\F_\frakp$ for $\frakp\neq T$, as well as its splitting and supersingular locus. More specifically, it is shown in \cite[Corollary 20]{BasBee} that if $\frakp\in \A$ is a prime of degree $n$, then a Drinfeld module $\phi$ as in \eqref{LegDrin} is supersingular at $\frakp$ if and only if $p_n\left(\frac{-\Delta}{T^q}\right)\equiv 0 \pmod \frakp$. The definition given for $p_n(x)$ in \cite{BasBee} is recursive. Combining the results from that paper with those of section \ref{period} we obtain an explicit closed form formula for those polynomials (which could be viewed as parallel to, albeit more complicated than, \eqref{Hp}), as well as a positive answer to the existence of a connection between formulas for periods (over $\C_\infty$) and polynomials encoding supersingularity  within certain families of Drinfeld modules.

It should be noted that any Drinfeld module is isomorphic (over an appropriate extension of its field of definition)  to one of the form \eqref{LegDrin}. In fact, since for such $\phi$ we have $\jmath(\phi)=\frac{(T+\Delta)^{q+1}}{\Delta}$, we see that for any given $\jmath\neq 0$ there are exactly $q+1$ distinct values of $\Delta$ which yield that $\jmath$. It is thus natural to also seek a description of the polynomial whose roots are the supersingular $\jmath$-invariants. (We note that the situation is similar to the classical case, for instance a generic $\jmath$-invariant can correspond to 6 different  $\lambda$-invariants in the family \eqref{Leg}.) 
It is well-known that the set $U_\frakp\subset \overline{\F}_\frakp $ of supersingular $\jmath$-invariants is finite, and in fact is a  $\textrm{Gal}(\F_\frakp^{(2)}/\F_\frakp)$ subset of $\F_\frakp^{(2)}$; the unique degree 2 extension of $\F_\frakp$ in $\overline{\F}_\frakp$. Furthermore 
we have $0\in U_\frakp$ if and only if $\deg(\frakp)$ is odd. It is customary to set
\[
ss_\frakp(x):=\prod_{\jmath\in U_\frakp\setminus\{0\}}(x-\jmath)\in \F_\frakp[x].
\]
Furthermore, the cardinality of $U_\frakp$ depends only on $\deg(\frakp)$ (for fixed $q$). Using the Chinese Remainder Theorem, we may consider polynomials $P_n(x)\in \A[x]$ whose reduction modulo any monic prime $\frakp$ of degree $n$ is $ss_\frakp(x)$ modulo $\frakp$. The polynomials $ss_\frakp(x)$ and $P_n(x)$ were studied by many authors. For instance, Cornelissen \cite{Cor99a,Cor99b} provided recurrences and congruences for $P_n(x)$, while Gekeler \cite{Gek11} recently proved certain congruences for $ss_\frakp$ in terms of companion polynomials to certain Eisenstein and so-called para-Eisenstein series.  We shall give closed form explicit formulas for $ss_\frakp$ modulo $\frakp$ (cf. Theorem \ref{ssformula}) that rely on relatively simple combinatorial objects called \emph{shadowed partitions} that were introduced by Papanikolas and the author in \cite{ElgPap}. These formulas are ``universal"  in that they depend only on $\deg(\frakp)$ and $q$.


\section{Periods in the Legendre Drinfeld family}\label{period}

Let $\Lambda\subset \C_\infty$ be a rank 2 $\A$-lattice. Attached to $\Lambda$ is the \emph{lattice exponential function} defined for $z\in \C_\infty$ by
\[
e_\Lambda(z):=z\prod_{0\neq \lambda \in \Lambda}\left(1-\frac{z}{\lambda}\right).
\]
We also have that $\Lambda/T\Lambda$ is a vector space of dimension 2 over $\F_q$, and it follows that $\prod_{\la \in \Lambda/T\Lambda}\left(x-e_\Lambda\left(\frac{\lambda}{T}\right)\right)$ is a polynomial of the form $x+g(\Lambda)x^q+\Delta(\Lambda) x^{q^2}$. Let $\phi^{\Lambda}$ be the Drinfeld module corresponding to $\Lambda$, which is given by $\phi_T^{\Lambda}:=T +g(\Lambda)\tau+\Delta(\Lambda)\tau^2$. It satisfies the functional equation 
\begin{equation}\label{fe}
e_\Lambda(Tz)=\phi^{\Lambda}_T(e_\Lambda(z)). 
\end{equation}
In analogy with the classical theory of elliptic curves, every rank $2$ Drinfeld module arises in this way, and the corresponding lattice $\Lambda_\phi$ of a given $\phi$ is called the \emph{period lattice} of $\phi$. We shall also denote the function satisfying \eqref{fe} for a given $\phi$ by $e_\phi$, and denote the formal inverse (which converges for $z$ sufficiently small) by $\log_\phi$.

We now recall some results of Papanikolas and the author from \cite{ElgPap} outlining a method of computing certain elements of $\Lambda_\phi$ from the knowledge of $\phi_T=T +g\tau+\Delta \tau^2$. For any such $\phi$, consider the set of $T$-torsion points given by $\ker(\phi_T)$; it is a $2$-dimensional $\F_q$ subspace of $\C_\infty$. Fix $0\neq\delta\in \C_\infty$, and set
\begin{equation}\label{fdel}
\mathcal{F_\delta}:=\{\textrm{All rank 2 Drinfeld modules $\phi$ over $\C_\infty$ such that }\F_q\delta\subset \ker(\phi_T)\}.
\end{equation}
We note that for $b\in \F_q^\times$, $\calF_{b\delta}=\calF_{\delta}$. Considering equal sets as equivalent, we see that any $\phi$ belongs to exactly two distinct sets $\calF_\delta$.
In order to guarantee convergence of certain series below, we shall consider the subfamily $\mathcal{F}_\delta^\star$ of $\mathcal{F}_\delta$ defined by
\begin{equation}\label{fstar}
\mathcal{F}_\delta^\star:=\left\{\phi \in \mathcal{F}_\delta: v(\jmath)<-q \textrm { and } v(\delta)=\frac{v(g)-v(\Delta)}{q^2-q}\right\}.
\end{equation}
We have the following analytic expression for certain periods of Drinfeld modules in $\mathcal{F}_\delta^\star$.

\begin{theorem}\textup{\cite[Theorem 6.3]{ElgPap}} \label{6.3}
Let $\phi \in \mathcal{F}_\delta^\star$ be given. Fix a choice of a $(q-1)$-st root of $\frac{T}{\Delta}$ and write
\[
c=\delta^{-1}\left(\frac{T}{\Delta}\right)^{\frac{1}{q-1}}.
\]
 Let $\beta_j$ be the coefficients of $\log_\phi(z)=\sum_{j=0}^\infty\beta_j z^{q^j}$. Set
\begin{equation}\label{a}
\begin{split}
\fraka_\delta(n)&:=T\sum_{j=0}^n \beta_j\delta^{q^j}, \textrm{ and }\\
\frakf(z)&:=\sum_{n=0}^\infty \fraka_\delta(n)z^{q^n}.
\end{split}
\end{equation}
Then the series $\frakf$ converges for $z=\delta^{-q}c$, and $\frakf(\delta^{-q}c)$ is a period of $\La_\phi$ with maximal valuation.

\end{theorem}


To motivate our definition of \emph{Legendre Drinfeld modules} we note that $\ker(\phi_T)$ can be viewed as the function field analogue of the $2$-torsion of an elliptic curve $E$. The Legendre family of elliptic curve corresponds to specifying the points $(x,y)\in\{(0,0), (1,0)\}$ to be nontrivial $2$ torsion points of curves in the family, leaving one free parameter corresponding to the third nontrivial $2$-torsion points. The function field analogue we propose is to fix a one-dimensional  subspace of $\ker(\phi_T)$ to be the simplest possible, namely $\F_q$ itself, again leaving ``one free parameter" determined by any $T$-torsion point not in $\F_q$.  This correspond to taking $\delta=1$ (or any element of $\F_q^\times$) in \eqref{fdel}, and we shall thus consider $\phi \in \calF_1$,  and set $\fraka_n=\fraka_1(n)$. Any Drinfeld module $\phi\in \calF_1$ is given by 
\[
\phi_T=T-(\Delta+T)\tau+\Delta\tau^2, \, \, \, \Delta\neq 0.
\]
Writing the associated logarithmic function as
\[
\log_\phi(z)=\sum_{i=0}^\infty \beta_i z^{q^i},
\]
we see that
\[
\fraka_n:=T\sum_{i=0}^n \beta_i.
\]
Since the function $\log_\phi$ satisfies the functional equation $T\log_\phi(z)=\log_\phi(\phi_T(z))$, it follows easily that the coefficients $\beta_n$ satisfy the recursion
\[
T\beta_n=T^{q^n}\beta_n+(-T-\Delta)^{q^{n-1}}\beta_{n-1}+\Delta^{q^{n-2}}\beta_{n-2},
\]
(we set $\beta_n=0$ for $n<0$). Consequently, the coefficients $\fraka_n$ satisfy
\[
\begin{split}
\fraka_n&=T^{q^n}\beta_n-\Delta^{q^{n-1}}\beta_{n-1}\\
&=T^{q^n-1}(\fraka_n-\fraka_{n-1})-\Delta^{q^{n-1}}T^{-1}(\fraka_{n-1}-\fraka_{n-2})\\
\end{split}
\]
multiplying by $T$ and re-arranging we get (setting $\fraka_{n}=0$ for $n<0$)
\begin{equation}\label{arecur}
-[n]\fraka_{n}=-\fraka_{n-1}\left(T^{q^{n}}+\Delta^{q^{n-1}}\right)+\Delta^{q^{n-1}}\fraka_{n-2},
\end{equation}
where $[n]:=T^{q^n}-T$. For future use we also set
\begin{equation}\label{dn}
\begin{split}
D_n&:=[n][n-1]^q[n-2]^{q^2}\cdots [1]^{q^{n-1}},\quad D_0:=1,\\
L_n&:= [n][n-1]\cdots[2][1],\quad L_0:=1.
\end{split}
\end{equation}
We will find it convenient to work with the rescaling of $\fraka_n$ given by
\[
\frakb_n:= \frac{L_n\fraka_n}{T^{1+q+\dots+q^n}}.
\]
multiplying \eqref{arecur} by $\frac{-L_{n-1}}{T^{1+q+\dots+q^{n}}}$, and setting (for typographical convenience)  $D:=\frac{\Delta}{T^q}$ we get
\begin{equation}\label{brecur}
\frakb_{n}=\left(1+D^{q^{n-1}}\right)\frakb_{n-1}+D^{q^{n-1}}(T^{1-q^{n-1}}-1)\frakb_{n-2},
\end{equation}
with $\frakb_{-1}=0, \frakb_0=1$. The first few examples are
\[
\begin{split}
&\frakb_0=1,\\
&\frakb_1=D+1,\\
&\frakb_2=D^{q+1}+D^qT^{1-q}+D+1,\\
&\frakb_3=D^{q^2+q+1}+D^{q^2+q}T^{1-q}+D^{q^2+1}T^{1-q^2}+D^{q^2}T^{1-q^2}+D^{q+1}+D^qT^{1-q}+D+1,\\
\end{split}
\]
We shall give an explicit combinatorial formula for $\frakb_n$, for which we need to introduce a little bit of notation first. For $S$ any finite subset of $\N=\{0,1,\dots\}$, define the subset $M(S)$ by
\begin{equation}
M(S):=\{i \in S: i-1\notin S\}.
\end{equation}
Furthermore, we associate to $S$ an integer ``\emph{weight}" $w(S)$ and a monomial $m(S)\in \A$ defined by 
\[
\begin{split}
w(S)&:=w(S;q):=\sum_{i\in S}q^i,\\
m(S)&:=m(S;q):=\prod_{i\in M(S)}T^{q^i-1}=T^{w(M(S))-|M(S)|},
\end{split}
\]
where the absolute value is used to denote the cardinality of the given set. Note that $m(\emptyset)=1$.
Finally, for $n>0$ we write $\N(n):=\{0,1,\dots, n-1\}$, $\N(0):=\emptyset$.

\begin{theorem}\label{bnthm}
For $n\geq 0$, the normalized period coefficient $\frakb_n$ is given by
\begin{equation}\label{bnformula}
\frakb_n=\sum_{S\subset \N(n)}\frac{D^{w(S)}}{m(S)}
\end{equation}
\end{theorem}
\begin{proof}
We proceed by induction on $n$. The formula clearly holds for $n=0,1$. Assume it holds for all values up to $n$, and note that $S\subset \N(n+1)$ with $n\notin S$ is equivalent to $ S\subset \N(n)$. It follows from \eqref{brecur} that
\[
\begin{split}
\frakb_{n+1}&=D^{q^n}(\frakb_n+(T^{1-q^n}-1)\frakb_{n-1})+\frakb_n\\
&=\sum_{S\subset \N(n+1), n\in S}\frac{D^{w(S)}}{m(S\setminus\{n\})}
+\sum_{S\subset \N(n+1), n\in S, n-1\notin S}\left(\frac{D^{w(S)}}{m(S)}- \frac{D^{w(S)}}{m(S\setminus\{n\})}\right)\\
&+\sum_{S\subset \N(n+1), n\notin S}\frac{D^{w(S)}}{m(S)}\\
&=\sum_{S\subset \N(n+1), n\in S, n-1\in S}\frac{D^{w(S)}}{m(S)}+\sum_{S\subset \N(n+1), n\in S, n-1\notin S} \frac{D^{w(S)}}{m(S)}+\sum_{S\subset \N(n+1), n\notin S}\frac{D^{w(S)}}{m(S)},\\
\end{split}
\]
where the last equality follows since if $\{n, n-1\}\subset S$, then $m(S)=m(S\setminus\{n\})$. The formula now follows immediately as the sets in the three summands in the last equation form a partition of $\N(n+1)$.
\end{proof}

\begin{corollary}
With the above notation, the coefficients of the series $\frakf(z)$ in Theorem \ref{6.3} for a Legendre Drinfeld module $\phi \in \calF_1$ are given by
\begin{equation}
\fraka_n= \frac{T^{1+q+q^2+\dots+q^n}}{L_n}\sum_{S\subset \N(n)}\frac{D^{w(S)}}{m(S)}.
\end{equation}
\end{corollary}

\section{Connections to the modular tower $X_0(T^n)$ and supersingular Drinfeld modules}\label{tower}

An important result in the theory of curves over finite fields is the Drinfeld-Vladut bound \cite{DrinVla}, which provides an asymptotic bound on the number of rational points on such curves as their genera grow to infinity. Several authors (see for example \cite{GarStic, Elk, BasBee} and the references within) addressed the nontrivial task of providing examples of curves of large genera over a finite field which are \emph{optimal} in the sense that they asymptotically attain the Drinfeld-Vladut bound. Explicit equations and computations on such curves have practical applications as they yield excellent linear error correcting Goppa codes \cite{TsfVla}.

One important source of such examples are the Drinfeld modular curves $X_0(T^n)$.  Elaborating on ideas of Elkies \cite{Elk}, Bassa and Beelen \cite{BasBee} recently provided recursive equations over $\F_q(T)$ for the tower $X_0(T^n))_{n\geq 2}$ whose reduction at various primes (different from $T$) are optimal over a quadratic extension of the base field. 
In the course of their study, they  introduced a sequence of polynomials $(p_n(x))_{n\geq -1}$ given by  $p_{-1}=0$, $p_0(x)=1$, and
\begin{equation}\label{precur}
p_{n+1}(x)=(x^{q^n}-1)p_n(x)+(1-T^{1-q^n})x^{q^n}p_{n-1}(x).
\end{equation}
As examples we have
\[
\begin{split}
p_1(x)&=x-1,\\
p_2(x)&=x^{q+1}-T^{1-q}x^q-x+1,\\
p_3(x)&=x^{q^2+q+1}-T^{1-q}x^{q^2+q}-T^{1-q^2}x^{q^2+1}+T^{1-q^2}x^{q^2}-x^{q+1}+T^{1-q}x^q+x-1,\\
p_4(x)&=x^{q^3+q^2+q+1}-\frac{x^{q^3+q^2+q}}{T^{q-1}}-\frac{x^{q^3+q^2+1}}{T^{q^2-1}}+\frac{x^{q^3+q^2}}{T^{q^2-1}}-\frac{x^{q^3+q+1}}{T^{q^3-1}}+\frac{x^{q^3+q}}{T^{q^3+q-2}}+\frac{x^{q^3+1}}{T^{q^3-1}}-\frac{x^{q^3}}{T^{q^3-1}}\\
&-x^{q^2+q+1}+\frac{x^{q^2+q}}{T^{q-1}}+\frac{x^{q^2+1}}{T^{q^2-1}}-\frac{x^{q^2}}{T^{q^2-1}}+x^{q+1}-\frac{x^q}{T^{q-1}}-x+1.\\
\end{split}
\]
These polynomials play an important role in determining the splitting locus of a certain Drinfeld modular tower (see  \cite[Section 4]{BasBee} for a precise statement and a few additional properties satisfied by these polynomials). Furthermore, they also proved the following result connecting these polynomials to certain supersingular Drinfeld modules.
 
\begin{proposition}\label{cor20}\textup {(Bassa and Beelen \cite[Corollary 20]{BasBee})}

The Drinfeld module given by 
\begin{equation}\label{bb}
\psi_T=T+(\Delta+T)\tau+\Delta\tau^2
\end{equation}
is supersingular at a prime $\frakp \in \A$ with $\deg(\frakp)=n$ if and only if 
\begin{equation}\label{ssp}
p_n\left(\frac{-\Delta}{T^q}\right) \equiv 0 \pmod \frakp.
\end{equation}
\end{proposition}
Combining this result with what we proved in Section \ref{period} we obtain the following.
\begin{theorem}
Let $\phi \in \calF_1$ be a Drinfeld module with $\Delta \in \A$ (so that we can reduce $\phi$ and $\Delta$ at any prime $\frakp\in \A$; we will continue denote the reductions by $\phi$ and $\Delta$, respectively). With the above notation, we have the following properties of $p_n$ and $\phi$.
\begin{enumerate}
\item[(i)] $p_n\left(\frac{-\Delta}{T^q}\right)=(-1)^n\frakb_n$. 

\item[(ii)] The polynomials $p_n(x)$ are given by
\begin{equation}
p_n(x)=\sum_{S\subset \N(n)}(-1)^{n-|S|}\frac{x^{w(S)}}{m(S)}.
\end{equation}

\item[(iii)] $\phi$ is supersingular at a prime $\frakp \in \A$ with $\deg(\frakp)=n$ if and only if \eqref{ssp} is satisfied.
\end{enumerate}
\end{theorem}
\begin{proof}
 The recurrence formula \eqref{precur} implies that $(-1)^n p_n\left(\frac{-\Delta}{T^q}\right)$ satisfies the exact same recurrence \eqref{brecur} for $\frakb_n$ with the same initial values, from which part (i) follows immediately. Part (ii) follows from substituting $x$ for $\frac{\Delta}{T^q}$ in \eqref{bnformula} and noting that $(-1)^{w(S)}=(-1)^{|S|}$. Part (iii) follows from Proposition \ref{cor20} since $\phi$ and $\psi$ have the same $\jmath$-invariant, and hence the same supersingularity behavior.
\end{proof}

\section{Supersingular $\jmath$-invariants}

In the previous section, we provided a closed formula for the polynomials $p_n$ whose roots corresponded to values $-\frac{\Delta}{T^q}$ for which a Drinfeld module as in \eqref{LegDrin} is supersingular. Since the isomorphism class of a Drinfeld module is determined by its $\jmath$-invariant rather than by $\Delta$, it is natural to also seek a description of the polynomial $ss_\frakp$ whose roots correspond to supersingular $\jmath$-invariants. We start by recalling some recent results of Gekeler (see \cite{Gek11} for more details and precise definitions). Throughout this section let $\phi$ be a generic Drinfeld module of rank 2 given by \eqref{phi}. Let $e_\phi(z)=\sum \alpha_n z^{q^n}$ and $\log_\phi(z)=\sum \beta_n z^{q^n}$ be the associated exponential and logarithmic functions, respectively. The coefficients $\beta_n$ are the Eisenstein series of weight $q^n-1$, whereas the coefficients $\alpha_n$ were baptized by Gekeler as \emph{para-Eisenstein} series of weight $q^n-1$. If we consider $\beta_n(\omega)$ and $\alpha_n(\omega)$ as the coefficients of the Drinfeld module $\phi^\Lambda$ corresponding to the lattice $\Lambda=\A+\A \omega$ with $\omega \in \C_\infty\setminus\K_\infty$, then $\beta_n$ and $\alpha_n$ are indeed modular forms of weight $q^n-1$ and type $0$ (the only type we will need to consider). Let $f$ be a modular form of weight $k>0$, it is well known that if $a>0$ and $0\leq b \leq q$ are the unique integers for which $k=a(q^2-1)+b(q-1)$, then there exists a unique polynomial $\varphi_f(x) \in \C_\infty[x]$ such that 
\[
f=\varphi_f(\jmath)\Delta^a g^b,
\]  
where $g(\omega)$ and $\Delta(\omega)$ are the unique normalized modular forms of weights $q-1$ and $q^2-1$, respectively. The polynomial $\varphi_f$ is called the \emph{companion polynomial} of $f$. Following Gekeler we set
\begin{equation}\label{mudef}
\begin{split}
&\mu_n(x):=D_n\varphi_{\alpha_n}(x),\\
&\gamma_n(x):=(-1)^nL_n\varphi_{\beta_n}(x),
\end{split}
\end{equation} 
with $D_n, L_n$ as in \eqref{dn}.The following result of Gekeler gives the precise connection between these polynomials and the supersingular polynomial.

\begin{proposition}\textup{\cite[Theorem 3.5(i)]{Gek11}}
Let $\frakp\in \A$ be a prime of degree $n$. Then
\begin{equation}\label{mugam}
\mu_n(x)\equiv \gamma_n(x)\equiv ss_\frakp(x) \pmod \frakp
\end{equation}
\end{proposition}
Gekeler studied various other properties of the polynomials $\mu_n(x)$ and $\gamma_n(x)$, including for example the properties of their set of exponents of nonvanishing coefficients (cf. \cite[Proposition 2.8]{Gek11}). We will be able to provide a complete description of these polynomials by relying on certain combinatorial objects called \emph{shadowed partitions}; these were introduced by Papanikolas and the author in \cite{ElgPap} and they play an important role for Drinfeld modules of any rank. We recall their basic definition for the rank 2 case now.

Recall that  a \emph{partition} of a set $S$ is a collection of subsets of $S$ that are pairwise disjoint,  and whose union is equal to $S$ itself. Also, if $S\subset \Z$, $j\in \Z$, then $S+j:=\{i+j: i\in S\}$.  For $ n \geq 1$, we set
\begin{multline}\label{prndefn}
P_2(n):=\bigl\{(S_1,\, S_2): S_i\subset \{0,\, 1\, \dots, n-1\},\\
 S_1, S_2, S_2+1\textnormal{ are distinct and form a partition of  $\{0,\, 1, \dots,\, n-1\}$} \bigr\}.
\end{multline}
We also set $P_2(0):=\{\emptyset\}$. Using these combinatorial objects, we will provide explicit formulas for $\mu_n(x)$ and $\gamma_n(x)$, and hence for the supersingular polynomials.

\begin{theorem}\label{ssformula}
Let $S$ be any subset of $\{0,1,\dots, n-1\}$ and set 
\begin{equation}
\begin{split}
D_n(S)&:=\prod_{i \in S} [n-i]^{q^{i}},\\
L_n(S)&:=(-1)^{|S|}[n]\prod_{0\neq i \in S} [i].
\end{split}
\end{equation}
Then with the above notation we have the following closed formula for the polynomials in \eqref{mudef}.
\begin{equation}\label{muformula}
\mu_n(\jmath)=\begin{cases}
& \sum_{(S_1,S_2) \in P_2(n)} \frac{D_n}{D_n(S_1\cup S_2)}\ \jmath^\frac{w(S_1)}{q+1} \, \textrm{ if $n$ is even},\\
&\sum_{(S_1,S_2) \in P_2(n)} \frac{D_n}{D_n(S_1\cup S_2)}\jmath^\frac{w(S_1)-1}{q+1} \, \textrm{ if $n$ is odd,}
\end{cases}
\end{equation}

\begin{equation}\label{gamformula}
\gamma_n(\jmath)=\begin{cases}
& \sum_{(S_1,S_2) \in P_2(n)} \frac{(-1)^nL_n}{L_n(S_1\cup S_2)}\ \jmath^\frac{w(S_1)}{q+1} \, \textrm{ if $n$ is even},\\
&\sum_{(S_1,S_2) \in P_2(n)} \frac{(-1)^nL_n}{L_n(S_1\cup S_2)}\jmath^\frac{w(S_1)-1}{q+1} \, \textrm{ if $n$ is odd.}
\end{cases}
\end{equation}

\end{theorem}
\begin{proof}
We have the following formulas from \cite[Theorems 3.1 and 3.3]{ElgPap}
\begin{equation}\label{alphabeta}
\begin{split}
\alpha_n&=\sum_{(S_1,S_2)\in P_2(n)} \frac{ g^{w(S_1)}\Delta^{w(S_2)}}{D_n(S_1\cup S_2)},\\
\beta_n&=\sum_{(S_1,S_2)\in P_2(n)} \frac{ g^{w(S_1)}\Delta^{w(S_2)}}{L_n(S_1\cup S_2)}.
\end{split}
\end{equation}
In order to evaluate $\phi_{\alpha_n}$ or $\phi_{\beta_n}$ from \eqref{alphabeta} we need to divide the right hand side by $\Delta^{1+q^2+\dots+q^{n-2}}$ if $n$ is even, or by $g\Delta^{q+q^3+\dots+q^{n-2}}$ if $n$ is odd. This yields \eqref{muformula} and \eqref{gamformula} after a simple computation.
\end{proof}





\begin{thebibliography}{999}

\bibitem{BasBee} A. Bassa and P. Beelen, \emph{Explicit equations for Drinfeld modular towers}, arxiv:1110.6076, 2011

\bibitem{Carlitz35} L. Carlitz, \emph{On certain functions connected with polynomials in a Galois field}, Duke Math. J. \textbf{1} (1935), 137--168.


\bibitem{Cor99a} G. Cornelissen, \emph{Deligne's congruence and supersingular reduction of Drinfeld modules}, Arch. Math. (Basel) \textbf{72} (1999), 346--353.

\bibitem{Cor99b} G. Cornelissen, \emph{Zeros of Eisenstein series, quadratic class numbers and supersingularity for rational function fields}, Math. Ann. \textbf{314} (1999), 175--196.

\bibitem{DrinVla} V. G. Drinfeld and S. G. Vladut, \emph{The number of points of an algebraic curve} (translated from the Russian paper in Funktsional. Anal. i Prilozhen), Functional Anal. Appl. \textbf{17} (1983), 53-54.

\bibitem{ElgOno} A. El-Guindy and K. Ono, \emph{Hasse invariants for the Clausen elliptic curves}, Ramanujan J. \textbf{31} (2013), 3-13.

\bibitem{ElgPap} A. El-Guindy and M.~A. Papanikolas,  \emph{Explicit formulas for Drinfeld modules and their periods}, J. Number Theory \textbf{133} (2013), 1864-1886.


\bibitem{Elk} N. D. Elkies, \emph{Explicit towers of Drinfeld modular curves}, Progress in Mathematics \textbf{202} (2001), 189--198.

\bibitem{GarStic} A. Garcia and H. Stichtenoth, \emph{Asymptotically good towers of function fields over finite fields}, C. R. Acad. Sci. Paris I \textbf{322}, no. 11 (1996), 1067-1070.

\bibitem{Gek11} E.-U.~Gekeler, \emph{Para-Eisenstein series for the modular group GL$(2,\F_q[T])$}, Taiwanese J. Math. \textbf{15}, no. 4 (2011), 1463-1475.

\bibitem{Gekeler88} E.-U. Gekeler, \emph{On the coefficients of Drinfeld modular forms},
Invent. Math. \textbf{93} (1988),  667--700.

J. Algebra \textbf{141} (1991), 187--203.

\bibitem{GossBook} D. Goss, \emph{Basic Structures of Function Field Arithmetic},
Springer, Berlin, 1996.


\bibitem{Hus}
D. Husem\"{o}ller, \emph{Elliptic Curves}, Springer Verlag, Graduate Texts in Mathematics, 111 (2004)







\bibitem{ThakurBook} D.~S. Thakur, \emph{Function Field Arithmetic}, World Scientific Publishing, River Edge, NJ, 2004.

\bibitem{TsfVla} M. A. Tsfaman and S. G. Vladut, \emph{Algebraic-Geometric Codes}, Dordrecht:Kluwer, 1991.

\bibitem{Wade46} L.~I. Wade, \emph{Remarks on the Carlitz $\psi$-functions}, Duke Math. J. \textbf{13} (1946), 71--78.

\end{thebibliography}
\end{document}